\documentclass[11pt]{article}
\usepackage{color,amsmath,mathrsfs,amssymb,accents}
\usepackage{amsthm}
\numberwithin{equation}{section}

\newtheorem{theorem}{Theorem}[section]
\newtheorem{lemma}[theorem]{Lemma}
\newtheorem{definition}[theorem]{Definition}
\newtheorem{remark}[theorem]{Remark}

\newcommand{\N}{\mathbb{N}}
\newcommand{\R}{\mathbb{R}}

\newcommand{\Leb}[1]{{\mathscr L}^{#1}} 

\newcommand{\e}{\varepsilon}

\renewcommand\div{\operatorname{div}}

 \newcommand{\tauV}{{\kern-3pt\tau}}

 \newcommand{\oVVVk}{\overline{\mbox{\boldmath$V$}}\kern-3pt}
 \newcommand{\tVVVk}{\tilde{\mbox{\boldmath$V$}}\kern-3pt}

 \newcommand{\eps}{\varepsilon}
 \newcommand{\negint}{{\int\negthickspace\negthickspace\negthickspace\negthickspace-}}

\begin{document}

\title{
Renormalized solutions to the continuity equation
\\ with an integrable damping term
}
\date{}

\author{   Maria Colombo\
   \thanks{Scuola Normale Superiore, Pisa. email: \textsf{maria.colombo@sns.it}}
 \and
Gianluca Crippa\    \thanks{Universit\"at Basel. email: \textsf{gianluca.crippa@unibas.ch}}
   \and
   Stefano Spirito\
     \thanks{GSSI - Gran Sasso Science Institute, L'Aquila. email:
   \textsf{stefano.spirito@gssi.infn.it}}
   }

\maketitle
\begin{abstract}
We consider the continuity equation with a nonsmooth vector field and a damping term. In their fundamental paper \cite{lions}, DiPerna and Lions proved that, when the damping term is bounded in space and time, the equation is well posed in the class of distributional solutions and the solution is transported by suitable characteristics of the vector field.
In this paper, we prove existence and uniqueness of renormalized solutions in the case of an integrable damping term, employing a new logarithmic estimate inspired by analogous ideas of Ambrosio, Lecumberry, and Maniglia \cite{alm}, Crippa and De Lellis \cite{crippade1} in the Lagrangian case.
\end{abstract}

\section{Introduction}

In this paper we consider the Cauchy problem for the continuity equation, namely
\begin{equation}\label{eqn:cont-damp}
\begin{cases}
\partial_t u(t,x)+ \nabla \cdot (b(t,x)u(t,x))= c(t,x)u(t,x)\\
u(0,x)=u_0(x)
\end{cases}
\end{equation}
where $(t,x)\in(0,T)\times\R^d$, $u\in\R$, $b\in\R^d$  and $c\in\R$. In analogy with fluid dynamics we call {\em damping} the term $c$. The continuity equation is a fundamental tool to study various nonlinear partial differential equations of the mathematical physics and it is often essential to deal with densities or with velocity fields which are not smooth. Starting from the papers of DiPerna and Lions \cite{lions} and Ambrosio \cite{ambrosio}, a huge literature has been developed in this direction (for an overview, see \cite{amcri-edi} and the references quoted therein).

The continuity equation is strictly related to the ordinary differential equation
\begin{equation}
\label{eqn:ODE}
\begin{cases}
&\partial_t X(t,x) = b(t,X(t,x)) \qquad \forall t \in (0,T)
\\
&X(0,x) =x
\end{cases}
\end{equation}
for $x\in \R^d$.
Indeed, assuming that $u_0$, $b$ and $c$ are smooth and compactly supported, we call $X :[0,T]\times\R^d\rightarrow\R^d$ the flow of $b$, namely the solution of the ODE \eqref{eqn:ODE}, and we set $JX(t,x) := \det(\nabla_x X(t,x)) \neq 0$. The map $X(t,\cdot)$ is a diffeomorphism and we denote by $X^{-1}(t,\cdot) $ its inverse. A solution of  \eqref{eqn:cont-damp} is then given in term of the flow $X$ by the following well-known explicit formula 
\begin{equation}\label{eqn:cont-damp-sol}
u(t,x)= \frac{u_0(X^{-1}(t,\cdot)(x))}{JX(t, X^{-1}(t,\cdot)(x))} \exp \Big({\int_0^t c(\tau, X(\tau, X^{-1}(t,\cdot)(x))) \, d\tau} \Big).
\end{equation}
Denoting with $f_\sharp \mu$ the pushforward of a Borel measure $\mu$ on $\R^d$ through a Borel function $f:\R^d \to \R^d$,  \eqref{eqn:cont-damp-sol} can be equivalently rewritten as
\begin{equation}\label{eqn:cont-damp-sol-pushforw}
u(t,\cdot) \Leb{d} = X(t, \cdot)_\sharp \left( u_0 \exp\Big({\int_0^t c(\tau,X(\tau, \cdot)) \, d\tau}\Big) \Leb{d} \right).
\end{equation}

If $c\in L^\infty( (0,T) \times \R^d)$, under suitable (regularity and growth) assumptions on the velocity field ensuring the existence and uniqueness of a Lagrangian flow, DiPerna and Lions \cite{lions} showed that \eqref{eqn:cont-damp-sol-pushforw} is the unique distributional solution of \eqref{eqn:cont-damp} with initial datum $u_0$.
At a very formal level, their strategy to prove uniqueness consists in considering the difference $u$ between two solutions with the same initial datum, which by linearity solves \eqref{eqn:cont-damp} with initial datum $0$, and multiplying the equation by $2u$. They obtain
\begin{equation}
\label{eqn:dipl-strategy}
\begin{split}
\frac{d}{dt}\int_{\R^d} u(t,x)^2 \, dx 
&= \int_{\R^d} (2c(t,x)-\nabla \cdot b(t,x))  u(t,x)^2 \, dx
\\&
\leq
 (2\|c(t,\cdot)\|_{L^\infty(\R^d)}+ \|\nabla \cdot b(t,\cdot)\|_{L^\infty(\R^d)})\int_{\R^d} u(t,x)^2 \, dx.
\end{split}
\end{equation}
They conclude thanks to Gronwall lemma that $\int_{\R^d} u(t,x)^2 \, dx =0$ for every $t\in [0,T]$, which implies uniqueness.

If $c \in L^1( (0,T) \times \R^d)$ then \eqref{eqn:cont-damp-sol-pushforw} does not make sense as distributional solution even in the simplest autonomous cases. For instance, let $b(t,x) = 0$, $u_0 = 1_{[0,1]^d}$, and $c\in L^1(\R^d)$. A solution of \eqref{eqn:cont-damp} is given by $u(t,x) = u_0(x) e^{tc(x)}$; however $u(t,\cdot)$
may not belong to $L^1_{\rm loc}(\R^d)$ due to the low integrability of $c$.
In this case \eqref{eqn:cont-damp-sol} is not a distributional solution of \eqref{eqn:cont-damp}.

We notice however that, if we assume $c\in L^1( (0,T) \times \R^d)$, the function $u$ defined in \eqref{eqn:cont-damp-sol} is almost everywhere pointwise defined because 
\begin{equation}
\label{eqn:buona-def-JX}
\int_{\R^d} \int_0^T c(\tau,X(\tau, x)) \, d\tau \, dx
\leq C
\int_0^T\int_{\R^d} c(\tau,x) \, d\tau \, dx <\infty
\end{equation}
since the flow is assumed to compress the Lebesgue measure in a controlled way.
In the following, we introduce a natural notion of renormalized solution of \eqref{eqn:cont-damp} (see Definition \ref{defn:renorm-sol}) following \cite{lions} and we prove that the function defined in \eqref{eqn:cont-damp-sol} is a renormalized solution of \eqref{defn:renorm-sol-eq}. Then we move to the more delicate problem of uniqueness with this weak notion of solution. 
Here a different estimate with respect to \eqref{eqn:dipl-strategy} is needed, since already the formal computation \eqref{eqn:dipl-strategy} fails if we assume lower summability than $L^\infty$ for the damping $c$. In analogy with the logarithmic estimates introduced by Ambrosio, Lecumberry and Maniglia \cite{alm}, Crippa and De Lellis \cite{crippade1} for solutions to the ODE \eqref{eqn:ODE}, we perform a logarithmic estimate for solutions of the PDE \eqref{eqn:cont-damp}.
As in the computation \eqref{eqn:dipl-strategy}, we consider the difference $u$ of two solutions with the same initial datum and we multiply \eqref{eqn:cont-damp} by $u/(\delta+u^2)$, where $\delta>0$ is fixed, and we obtain
\begin{equation}
\label{eqn:ccs-strategy}
\begin{split}
\frac{d}{dt} \int_{\R^d} \! \log\Big(1+ \frac{ u(t,x)^2}{\delta} \Big) \, dx 
&= \int_{\R^d}\! \nabla \cdot b(t,x) \log\Big(1+ \frac{ u(t,x)^2}{\delta} \Big) \, dx\\
&+ \int_{\R^d}\!(c(t,x)-\nabla\cdot b(t,x)) \frac{ u(t,x)^2}{\delta + u(t,x)^2}  \, dx\\
&\leq\|\nabla \cdot b(t,\cdot)\|_{L^\infty(\R^d)} \int_{\R^d} \log\Big(1+ \frac{ u(t,x)^2}{\delta} \Big) \, dx\\
&+2\int_{\R^d} |c(t,x)|+|\nabla\cdot b(t,x)| \, dx.
\end{split}
\end{equation}
By Gronwall lemma we deduce that for every $t\in [0,T]$
\begin{equation*}
\begin{split}
 \int_{\R^d} \log\Big(1+ \frac{ u(t,x)^2}{\delta} \Big) \, dx 
&\leq
\exp\Big( \int_0^T \|\nabla \cdot b(t,\cdot)\|_{L^\infty(\R^d)} \, dt \Big)\\
&\hspace{2em}\cdot \int_0^T \int_{\R^d} 2(|c(t,x)|+|\nabla\cdot b(t,x)|) \, dx\, dt
;
\end{split}
\end{equation*}
letting finally $\delta$ go to $0$, since the right-hand side is independent on $\delta$ we obtain that $u(t, \cdot ) = 0$.
A justification of the estimate \eqref{eqn:ccs-strategy} in a nonsmooth setting requires some work, as in \cite{lions} to justify \eqref{eqn:dipl-strategy}. First, one needs to prove that the difference of renormalized solutions is still renormalized, which is not an automatic consequence of the linearity of the equation and of the theory of renormalized solutions. Moreover, to allow general growth conditions on $b$, one would like to localize the estimate. In \cite{lions}, general growth conditions were considered by means of a cutoff function and by a duality argument. Instead, we refine the estimate \eqref{eqn:ccs-strategy} by means of a decaying function.

\smallskip
The plan of the paper is the following. In Section~\ref{sec:esuniq} we introduce the notions of regular Lagrangian flow and of renormalized solution; then we state our existence and uniqueness result.
Sections~\ref{sec:ex} and~\ref{sec:un} are devoted to the proof of the main theorem.
In Section~\ref{sec:bmo} we show how we can use a logarithmic estimate like in \cite{crippade1} 
to prove an uniqueness result for the continuity equation when $\nabla \cdot b$ is a bounded mean oscillation function, giving a new proof of a result in~\cite{bmo}.

\smallskip
\noindent {\bf Acknowledgements.} 
This research has been partially supported by the SNSF grants 140232 and 156112. The first author acknowledges the kind hospitality of the University of Basel, where most of this work has been done. 
This work has been started while the third author was a PostDoc at the Departement Mathematik und Informatik of the Universit\"at Basel. He would like to thank the department for the hospitality and the support.

\section{Preliminaries and statement of the main result}\label{sec:esuniq}

We denote by $B_r(x) \subseteq \R^d$ the open ball of centre $x\in \R^d$ and radius $r>0$, shortened to $B_r$ if $x=0$.
In the case of a smooth, divergence free vector field $b$, the solution to the equation \eqref{eqn:cont-damp}, given by the explicit formula \eqref{eqn:cont-damp-sol} with $JX(t,x) = 1$, is obtained by transporting the initial datum $u_0$ along the flow of the vector field $b$, together with a correction due to the damping term $c$. To obtain a similar statement in the nonsmooth setting, we recall the notion of Lagrangian flow.
\begin{definition}\label{defn:rlf}
Let $T>0$ and  $b: (0,T) \times \R^d \to \R^d$ a Borel, locally integrable vector field. We say that $X:\R^d\times [0,T]\to\R^d$ is a regular Lagrangian flow of $b$ if the following two properties hold: 
\begin{itemize}
\item[(i)] for $\Leb{d}$-a.e. $x\in \R^d$, $X(\cdot,x)\in AC([0,T];\R^d)$ and solves the ODE $\dot x(t)=b(t,x(t))$ $\Leb{1}$-a.e.
in $(0,T)$, with the initial condition $X(0,x)=x$;
\item[(ii)] there exists a constant $C=C(X)$ satisfying $X(t,\cdot)_\#\Leb{d}\leq C\Leb{d}$.
\end{itemize}
\end{definition}

The well-celebrated papers of DiPerna and Lions \cite{lions} and of Ambrosio \cite{ambrosio} provide existence and uniqueness of the regular Lagrangian flow assuming local Sobolev or $BV$ regularity of $b$, boundedness of the distributional divergence $\nabla \cdot b$, and some growth conditions on $b$.

\begin{theorem}
Let $b\in L^1((0,T); BV_{\rm loc} (\R^d; \R^d))$ be a vector field that satisfies a bound on the divergence 
$\nabla \cdot b \in L^1((0,T); L^\infty(\R^d))$ and the growth condition
$$
\frac{|b(t,x)|}{1+|x|} \in L^1((0,T); L^1(\R^d))+ L^1((0,T); L^\infty(\R^d)).
$$
Then there exists a unique regular Lagrangian flow $X$ of $b$.
\end{theorem}

\begin{remark}\label{rmk:inv-compr}
Under the assumptions of the previous theorem, thanks to the bilateral bound on $\nabla \cdot b$, the map $X(t, \cdot)$ is a.e. invertible for every $t\in [0,T]$. We denote by $X^{-1}(t,\cdot)$ the inverse map. Moreover, the compressibility constant $C(X)$ in Definition~\ref{defn:rlf} (ii) can be chosen as $\exp \big( \int_0^T \|\nabla \cdot b(t,\cdot) \|_{L^\infty(\R^d)} \, dt \big)$.
\end{remark}

When the vector field $b$ is divergence-free, the Jacobian of the flow is equal to $1$ in the explicit solution \eqref{eqn:cont-damp-sol} of \eqref{eqn:cont-damp}. Instead, when the vector field $b$ is not divergence-free, the Jacobian of the flow appears in \eqref{eqn:cont-damp-sol}. In the smooth setting, the Jacobian is defined as $JX(t,x) = \det(\nabla_x X(t,x))$, and satisfies the differential equation 
$$\partial_t JX(t,x) = JX(t,x) \nabla \cdot b( t, X(t,x))
\qquad
\forall (t,x) \in (0,T) \times \R^d.$$
In the nonsmooth setting, we define the Jacobian through an explicit formula; we will see in Lemma~\ref{lemma:jac} that this object satisfies a change of variable formula.
\begin{definition}\label{defn:jac}
Let $T>0$, $b$, $X$ as in Definition~\ref{defn:rlf}. Assume moreover that $\nabla \cdot b \in L^1((0,T); L^1_{\rm loc}(\R^d))$.
We define the Jacobian of $X$ as the measurable function $JX:(0,T) \times \R^d \to \R^d$ given by
$$JX(t,x) = \exp \Big( {\int_0^t \nabla \cdot b(s,X(s,x)) \, ds} \Big).$$
\end{definition}
Thanks to the compressibility condition (ii) in the definition of regular Lagrangian flow and to the local integrability of $\nabla \cdot b$, a computation like \eqref{eqn:buona-def-JX} shows that $JX$ is well defined and absolutely continuous in $[0,T]$ for a.e. $x\in \R^d$.

\smallskip

We present now a notion of solution of \eqref{eqn:cont-damp} which does not even require local integrability of $u$ and was first introduced in \cite{lions}. In the sequel all the functions involved will be defined up to a set of measure zero. 
%
\begin{definition}\label{defn:renorm-sol} Let $u_0: \R^d \to \R$ be a measurable  function, let $b\in L^1_{\rm loc} ((0,T) \times \R^d; \R^d)$ be a vector field such that $\nabla \cdot b\in L^1_{\rm loc} ((0,T) \times \R^d)$ and let $c\in L^1_{\rm loc} ((0,T) \times \R^d)$.
A measurable function $u:[0,T] \times \R^d \to \R$ is a renormalized solution of \eqref{eqn:cont-damp} if for every function $\beta:\R\to \R$ satisfying
\begin{equation}
\label{defn:beta}
\beta \in C^1\cap L^\infty(\R), \qquad \beta'(z) z \in L^\infty(\R), \qquad \beta(0)=0
\end{equation}
 we have that
 \begin{equation}
 \label{defn:renorm-sol-eq}
 \partial_t \beta (u) + \nabla \cdot (b\beta(u)) + \nabla \cdot b \big(u \beta'(u)-\beta(u)\big) = cu \beta'(u)
 \end{equation} in the sense of distributions, namely 
  for every $\phi \in C^{\infty}_c ( [0,T) \times \R^d)$
 \begin{equation}
 \label{defn:renorm-sol-tested-st}
 \begin{split}
&\hspace{-3em}\int_{\R^d} \phi(0,x) \beta (u_0)\, dx +
 \int_0^T \! \int_{\R^d} [\partial_t \phi + \nabla\phi \cdot b ]\beta(u) \, dx \, dt  \,+
  \\&\int_0^T \!\! \int_{\R^d} \!\phi \Big[ \nabla \cdot b \big(\beta(u)-u \beta'(u)\big) +  cu \beta'(u) \Big] \, dx \, dt =0.
  \end{split}
\end{equation}
 \end{definition}
The second assumption in \eqref{defn:beta} is exploited to give a distributional meaning to the right-hand side of \eqref{defn:renorm-sol-eq}, which becomes locally integrable despite the lack of integrability of $u$.

\begin{remark}\label{rmk:renorm-operativa}
In Definition~\ref{defn:renorm-sol}, we can equivalently test equation \eqref{defn:renorm-sol-eq} with compactly supported space functions $\varphi$; in other words, \eqref{defn:renorm-sol-tested-st} holds if and only if 
for every $\varphi \in C^{\infty}_c ( \R^d)$ the function $\int_{\R^d} \varphi(x) \beta (u(t,x))\, dx
$ coincides a.e. with an absolutely continuous function $\Gamma(t)$ such that
$\Gamma(0)=\int_{\R^d} \varphi(x) \beta (u_0(x))\, dx $
   and for a.e. $t\in [0,T]$
 \begin{equation}
 \label{defn:renorm-sol-tested}
 \begin{split}
 \frac{d}{dt} \Gamma(t)
 &= \int_{\R^d} \nabla\varphi \cdot b\beta(u) \, dx  +\int_{\R^d} \varphi(x) \Big[ \nabla \cdot b\big(\beta(u)-u \beta'(u)\big) +  cu \beta'(u) \Big] \, dx .
\end{split}
\end{equation}
This follows by the choice $\phi(t,x) = \varphi(x) \eta(t)$ in \eqref{defn:renorm-sol-tested-st} with $\eta \in C^\infty_c([0,T))$; by the density of the linear span of these functions in $C^\infty_c([0,T) \times \R^d)$, it is possibile to deduce the equivalence (see also \cite[Section 8.1]{amgisa}).
Notice moreover that, with a standard approximation argument, we are allowed to use every Lipschitz, compactly supported test function $\varphi: \R^d \to \R$ as a test function for the computation \eqref{defn:renorm-sol-tested}.
\end{remark}
The main result of this paper provides existence and uniqueness of renormalized solutions to the continuity equation with integrable, unbounded damping.
\begin{theorem}\label{thm:main}
Let $b\in L^1((0,T); BV_{\rm loc} (\R^d; \R^d))$ be a vector field that satisfies a bound on the divergence 
$\nabla \cdot b \in L^1((0,T); L^\infty(\R^d))$ and the growth condition
\begin{equation}
\label{eqn:growth-b}
\frac{|b(x)|}{1+|x|} \in L^1((0,T); L^1(\R^d))+ L^1((0,T); L^\infty(\R^d)).
\end{equation}
Let 
$$c\in L^1( (0,T) \times \R^d)$$
and let $u_0:\R^d \to \R$ be a measurable function.
Then there exists a unique renormalized solution $u:[0,T] \times \R^d \to \R$ of \eqref{eqn:cont-damp} starting from $u_0$ and it is given by the formula 
\begin{equation}\label{eqn:cont-damp-sol-thm}
u(t,x)= \frac{u_0(X^{-1}(t,\cdot)(x))}{JX(t, X^{-1}(t,\cdot)(x))} \exp\Big({\int_0^t c(\tau, X(\tau, X^{-1}(t,\cdot)(x))) \, d\tau}\Big).
\end{equation}
\end{theorem}

\begin{remark}
The same statement holds for vector fields $b$ satisfying other local regularity assumptions than $BV$; more precisely, Theorem~\ref{thm:main} holds for every $b$ such that every bounded, distributional solution of the continuity equation $\partial_t u + \nabla \cdot (bu) = f\in L^{1}_{\rm loc}$ is renormalized.
To see some classes of vector fields other than BV which satisfy this assumption, the interested reader may refer to \cite{amcri-edi} and to the references quoted therein.
\end{remark}

\section{Existence}\label{sec:ex}
To prove existence in Theorem~\ref{thm:main}, we show by explicit computation that  \eqref{eqn:cont-damp-sol} provides a renormalized solution to \eqref{eqn:cont-damp}.
In the case of a divergence-free vector field, the flow $X(t, \cdot)$ is measure preserving and \eqref{eqn:cont-damp-sol-thm} can be rewritten as
$$u(t,x)= {u_0(X^{-1}(t,\cdot)(x))}\exp \Big({\int_0^t c(\tau, X(\tau, X^{-1}(t,\cdot)(x))) \, d\tau} \Big).
$$
An easy computation shows that this function is a renormalized solution of \eqref{eqn:cont-damp}:
\begin{equation*}
\begin{split}
&\frac{d}{dt} \int_{\R^d} \varphi \beta(u) \, dx 
=
\frac{d}{dt} \int_{\R^d} \varphi(X) \beta(u(t,X)) \, dx 
\\
&=
\int_{\R^d} \Big[ \nabla \varphi(X)\cdot b(t,X) \beta(u(t,X)) 
+ \varphi(X)  \beta'(u(t,X)) u(t,X) c(t,X) \Big] \, dx 
\\
&=
\int_{\R^d} \Big[ \nabla \varphi \cdot b \beta(u) 
+ \varphi  \beta'(u) u c \Big] \, dx, 
\end{split}
\end{equation*}
(compare with \eqref{defn:renorm-sol-tested}). Note that it the above calculation is has been used that when the representation formula is considered along the flow it holds that  
$$
\frac{d}{dt}\left(u(t,X)\right)=u(t,X)c(t,X).$$
The computation can be made rigorous thanks to the absolute continuity of $X(\cdot, x)$.

The following lemma, regarding time regularity of the Jacobian of regular Lagrangian flows, is useful in the proof when $b$ is not divergence-free.
\begin{lemma}[Properties of the Jacobian]\label{lemma:jac}
Let $b$ as in Theorem~\ref{thm:main} and let $X$ be the regular Lagrangian flow of $b$. Then, the function $JX$ in Definition \ref{defn:jac} is in $L^1((0,T);L^\infty(\R^d))$ and for every $t>0$ and every $\phi\in L^1(\R^d)$ satisfies the following change of variable formula:
\begin{equation}
\label{eqn:c-d-var}
\int_{\R^d} \phi(X(t,x) ) JX(t,x) \, dx = \int_{\R^d} \phi(x) \, dx.
\end{equation}
Moreover, $e^{-L} \leq JX \leq e^L$ with $L= {\int_0^T \|\nabla \cdot b(t,\cdot) \|_{L^\infty(\R^d)} \, dt}$,  $JX (\cdot, x) $ and $JX^{-1}(\cdot,x )$ are absolutely continuous in $[0,T]$ for a.e. $x \in \R^d$ and satisfy
\begin{equation}
\label{eqn:JX}
\partial_t JX(t,x) = JX(t,x) \nabla \cdot b( t, X(t,x))
\qquad
\mbox{for a.e. } t\in (0,T)
,
\end{equation}
\begin{equation}
\label{eqn:JX-1}
\partial_t \left[\frac{1}{JX(t,x)}\right] = -\left(\frac{1}{JX(t,x)}\right)\nabla \cdot b( t, X(t,x))
\qquad \mbox{for a.e. } t\in (0,T)
.
\end{equation}
\end{lemma}
\begin{proof}
{\it Step 1: approximation with smooth vector fields.} Let us approximate the vector field $b$ by convolution. In particular let $b^\eps$ be the convolution between $b$, extended to $0$ in $(\R\setminus [0,T]) \times \R^d$, and a kernel of the form $\eps^{-d-1} \rho_1(t/\eps) \rho_2(x/\eps)$, where $\rho_1\in C^\infty_c(\R)$ and $\rho_2\in C^\infty_c(\R^d)$ are standard convolution kernels,
so that
\begin{equation}
\label{eqn:div-eps}
\| \nabla \cdot b^\eps(t, \cdot) \|_{L^\infty(\R^d)} \leq 
\int_\R \rho_1(t' )\| \nabla \cdot b(t- \eps t', \cdot) \|_{L^\infty(\R^d)} \, dt'.
\end{equation}
Let $X^\eps\in C^\infty([0,T] \times \R^d ; \R^d)$ be the flow of $b^\eps$; for every $t>0$ the function $X^\eps(t,\cdot)$ is a diffeomorphism of $\R^d$ and, setting $JX^\eps(t,x) = \det \nabla_x X^\eps(t,x)$, we have the change of variable formula
\begin{equation}
\label{eqn:cdvar-eps}
\int_{\R^d} \phi(X^\eps(t,x) ) JX^\eps(t,x) \, dx = \int_{\R^d} \phi(x) \, dx  \qquad \forall \phi \in C_c(\R^d).
\end{equation}
Moreover for every $x\in \R^d$ the function $JX^\eps(\cdot,x)$ solves the ODE
$$\begin{cases}
\partial_t JX^\eps(t,x) = JX^\eps (t,x) \nabla \cdot b^\eps( t, X^\eps(t,x))
\qquad&
\mbox{for any } t\in (0,T)
\\
JX^\eps(0,x) = x,&
\end{cases}$$
hence it is given by the expression
$$JX^\eps (t,x) = \exp \Big( {\int_0^t \nabla \cdot b^\eps(s,X^\eps(s,x)) \, ds} \Big) \qquad \forall (t,x) \in [0,T] \times \R^d.$$
Integrating \eqref{eqn:div-eps} in $(0,T)$, we find that 
\begin{equation}
\label{eqn:eL}
e^{-L} \leq JX^\eps \leq e^L \qquad \forall(t,x) \in [0,T] \times \R^d \qquad \mbox{with }L= {\int_0^T \|\nabla \cdot b(t,\cdot) \|_{L^\infty(\R^d)} \, dt}.
\end{equation}

{\it Step 2: pointwise convergence of Jacobians.}
We show that, up to a subsequence (not relabeled) in $\eps$, for a.e. $x\in \R^d$
\begin{equation}\label{eq:J}
\lim_{\eps\to 0} JX^\eps (t,x) = JX(t,x) \qquad
\mbox{for every } t\in (0,T),
\end{equation}where $JX$ is defined in Definition~\ref{defn:jac}.

To this end, let us first prove that, up to a subsequence (not relabeled),
\begin{equation}
\label{eqn:conv-bX}
\lim_{\eps \to 0} \nabla \cdot b^\eps(t,X^\eps(t,x)) = \nabla \cdot b(t,X(t,x))
\qquad \mbox{in }L^1_{\rm loc}([0,T] \times \R^d).
\end{equation}
By the stability of regular Lagrangian flows (see \cite{ambrosio,lions,crippade1} or \cite[Section 5]{amcri-edi}), for every $t\in [0,T]$
we have that, up to a subsequence (not relabelled)
\begin{equation}
\label{eqn:conv-X}
\lim_{\eps \to 0}X^\eps(t,x) = X(t,x)
\qquad \mbox{pointwise for a.e. }(t,x)\in [0,T] \times \R^d.
\end{equation}

Let us consider $r>0$ and let us prove the convergence in \eqref{eqn:conv-bX} in $[0,T] \times B_r$. Let $R>0$ and $\eta>0$ to be chosen later. The estimate on superlevels in \cite[Proposition 3.2]{crippade1}, which depends on the growth assumptions \eqref{eqn:eL} and on the compressibility of the flows, implies that
\begin{equation}
\label{eqn:superlevels}
\Leb{d}(\{x \in B_r : X^\e(t, x) \in \R^d \setminus B_R\}) \leq g(R,r),
\end{equation}¥ for a function $g(R,r)$ which converges to $0$ as $R \to \infty$ for every $r>0$ (and it is independent on $\e$ and $t$). The analogous of \eqref{eqn:superlevels} holds also with $X$ in place of $X^\eps$.

By Egorov theorem, there exists a measurable set $E\subseteq [0,T] \times B_R$ of small measure $\Leb{d+1}(E) \leq \eta$ such that 
\begin{equation}\label{eq:egorov}
\lim_{\eps \to 0 }\nabla \cdot b^\eps=\nabla \cdot b  \qquad \mbox{uniformly in } \big( [0,T] \times B_R \big) \setminus E.
\end{equation}
As a consequence, $ \nabla \cdot b(t, \cdot)$ is continuous on $ \big([0,T] \times B_R \big)\setminus E$. Let us consider $E^t$ to be the intersection of $E$ with $\{t\} \times \R^d$. Letting
$$E^t_{\eps, R} = \{x \in B_r: X^\eps(t,x) \in E^t \cup (\R^d \setminus B_R)\} \cup  \{x \in B_r:  X(t,x) \in E^t\cup (\R^d \setminus B_R)\},$$
we have that
\begin{equation}\label{eqn:divbX-conv}
\begin{aligned}
\int_0^T\int_{B_r} |\nabla \cdot b^\eps(t,X^\eps) - \nabla \cdot b(t,X) | \, dx \, dt  &\leq
\int_0^T\int_{B_r \setminus E^t_{\eps,R}} |\nabla \cdot b^\eps(t,X^\eps) - \nabla \cdot b(t,X) | \, dx \, dt
\\
&+
\int_0^T \Leb{d}(E^t_{\eps,R}) \|\nabla \cdot b^\eps (t, \cdot)\|_{L^\infty(\R^d)}\,dt\\
&+
\int_0^T \Leb{d}(E^t_{\eps,R}) + \|\nabla \cdot b (t, \cdot)\|_{L^\infty(\R^d)}) \, dt
\end{aligned}
\end{equation}
The second and the third term in the right-hand side of \eqref{eqn:divbX-conv} can be estimated uniformly in $\eps$ thanks to the compressibility of $X^\eps(t,\cdot)$ and $X(t,\cdot)$, which is less or equal, in both cases, than $ e^L$ thanks to \eqref{eqn:eL} and Remark~\ref{rmk:inv-compr}. More precisely
\begin{equation*}
\begin{aligned}
\Leb{d}(\{ X^\eps(t,\cdot) \in E^t \cup (\R^d \setminus B_R)\} )
&\leq
\Leb{d}(\{ X^\eps(t,\cdot) \in E^t \} )
+
\Leb{d}(\{ X^\eps(t,\cdot) \in \R^d \setminus B_R\} )
\\
& \leq e^L \Leb{d}(E^t)+ g(R,r) 
\end{aligned}
\end{equation*}
and a similar computation holds for the set $\{ X(t,\cdot) \in E^t \cup (\R^d \setminus B_R)\}$, so that overall
$$\Leb{d}(E^t_{\eps, R}) \leq 2e^L \Leb{d}(E^t) + 2g(R,r).$$

Thanks to 
\eqref{eqn:divbX-conv}, it implies that
\begin{equation*}
\begin{split}
\hspace{-1em}\int_0^T\int_{B_r} |\nabla \cdot b^\eps(t,X^\eps) - \nabla \cdot b(t,X) | \, dx \, dt  &
\leq
\int_0^T\int_{B_r \setminus E^t_{\eps,R}} |\nabla \cdot b^\eps(t,X^\eps)- \nabla \cdot b(t,X) | \, dx \, dt
\\
&+ 2
e^L \int_0^T \Leb{d}(E_t) \|\nabla \cdot b (t, \cdot)\|_{L^\infty(\R^d)}\,dt\\
&+2 e^L\int_0^T\Leb{d}(E_t) \|\nabla \cdot b^{\eps} (t, \cdot)\|_{L^\infty(\R^d)} \, dt\\
&+4
 g(R,r) \int_0^T \|\nabla \cdot b (t, \cdot)\|_{L^\infty(\R^d)} \, dt
 \end{split}
 \end{equation*}
 which can be written as follows:
 \begin{equation}
\label{eqn:divbX-conv-new}
\begin{split}
\int_0^T\int_{B_r} |\nabla \cdot b^\eps(t,X^\eps) - \nabla \cdot b(t,X) | \, dx \, dt &\leq
 \int_0^T\int_{B_r \setminus E^t_{\eps,R}} |\nabla \cdot b^\eps(t,X^\eps) - \nabla \cdot b(t,X) | \, dx \, dt\\
 &+4
 g(R,r) \int_0^T \|\nabla \cdot b (t, \cdot)\|_{L^\infty(\R^d)} \, dt\\
&+ 2
e^L \int_{E} \|\nabla \cdot b (t, \cdot)\|_{L^\infty(\R^d)}\,dx\,dt\\
&+2 e^L\int_{E} \|\nabla \cdot b^{\eps} (t, \cdot)\|_{L^\infty(\R^d)}\,dx \, dt.
\end{split}
\end{equation}
The first term in \eqref{eqn:divbX-conv-new} converges to $0$ as $\eps \to 0$ because $\nabla \cdot b^\eps(t,X^\eps)$ converges pointwise to $\nabla \cdot b(t,X)$ in $B_r \setminus E^t_{\eps,R}$ and $\nabla \cdot b$ is continuous on $E^t$:
\begin{equation}
\begin{aligned}
&|\nabla \cdot b^\eps(t,X^\eps) - \nabla \cdot b(t,X) |
\leq
|\nabla \cdot b^\eps(t,X^\eps) - \nabla \cdot b(t,X^\eps) |
+|\nabla \cdot b(t,X^\eps) - \nabla \cdot b(t,X) |
\\
&\hspace{10em}\leq
\|\nabla \cdot b^\eps(t,\cdot) - \nabla \cdot b(t,\cdot) \|_{L^\infty(B_R \setminus E^t)}
+|\nabla \cdot b(t,X^\eps) - \nabla \cdot b(t,X) |.
\end{aligned}
\end{equation}
The second term goes to $0$ because $g(R,r)\to 0$ for $R\to\infty$. The last terms, in turn, converge to $0$ as $\eta \to 0$, where $\eta$ has been chosen in \eqref{eq:egorov} and is independently on $\eps$, by the absolute continuity of the Lebesgue integral.
Indeed, each function is dominated by 
 $$t \mapsto \|\nabla \cdot b^\eps (t, \cdot)\|_{L^\infty(\R^d)} + \|\nabla \cdot b (t, \cdot)\|_{L^\infty(\R^d)} \leq (\eps^{-1}\rho_1(\cdot/\eps)) \ast \|\nabla \cdot b (t, \cdot)\|_{L^\infty(\R^d)} + \|\nabla \cdot b (t, \cdot)\|_{L^\infty(\R^d)}$$
and the last function converges in $L^1([0,T])$ to  $2\|\nabla \cdot b (t, \cdot)\|_{L^\infty(\R^d)}$, so that we can take the limit in the right-hand side of \eqref{eqn:divbX-conv-new} by the absolute continuity of the Lebesgue integral. 
Finally, choosing first $R$ and $\eta$ small enough, and then letting $\eps$ go to $0$  in \eqref{eqn:divbX-conv-new}, we find \eqref{eqn:conv-bX}.
By \eqref{eqn:conv-bX}, up to a subsequence, for a.e. $x\in \R^d$, $ \nabla \cdot b^\eps(t,X^\eps(t,x))$ converges to $\nabla \cdot b(t,X(t,x))$ in $L^1([0,T])$. Hence for a.e. $x$ we deduce \eqref{eq:J}. 

{\it Step 3: conclusion.}
Let us fix $t>0$ and $\phi \in C_c(B_R)$ with $R>0$. We take the limit as $\eps$ goes to $0$ in \eqref{eqn:cdvar-eps} to get \eqref{eqn:c-d-var}. 
More precisely, to show that the limit of \eqref{eqn:cdvar-eps} is \eqref{eqn:c-d-var} we estimate the difference of the two terms by adding and subtracting $\phi(X)JX^{\e}$ and using the bound on $JX^\e$ given by \eqref{eqn:eL}
\begin{equation*}
\left|\int_{\R^d}(\phi(X^{\e})JX^{\e} -\phi(X)JX) \, dx \right|\leq
\int_{\R^d}|\phi(X)| |JX^{\e}-JX| \, dx+ e^L\int_{\R^d}|\phi(X^{\e})-\phi(X)| \, dx.
\end{equation*}
The first term goes to $0$ as $\e \to 0$ by \eqref{eq:J} and the dominated convergence theorem, since the functions are nonzero only on the set $\{ x: X(t,x) \in B_R\}$ and this set has finite measure.

Regarding the second term, for every $\tilde R>0$ we have
\begin{equation*}
\begin{aligned}
\int_{\R^d}|\phi(X^{\e})-\phi(X)| \, dx &\leq 2 \|\phi \|_{L^\infty} \Leb{d}( \{ x\notin B_{\tilde R}: X(t,x) \in B_R \mbox{ or } X^\e(t,x) \in B_R\})\\ 
&+ \int_{B_{\tilde R}}|\phi(X^{\e})-\phi(X)| \, dx.
\end{aligned}
\end{equation*}
By choosing $\tilde R$ sufficiently big, the first term can be made as small as we want independently on $\e$ thanks to the estimate on superlevels in \cite[Proposition 3.2]{crippade1} (see also \eqref{eqn:superlevels}). Finally, letting $\e \to 0$ in the second term with $\tilde R$ fixed, we obtain that it converges to $0$ by dominated convergence.
%
%
Hence, \eqref{eqn:c-d-var} holds true for every $\phi \in C_c(\R^d)$. Then we approximate every $\phi\in L^1(\R^d)$ with compactly supported, continuous functions $\{\phi_n\}_{n\in \N}$ and we take the limit in \eqref{eqn:c-d-var} applied to $\phi_n$.
The left-hand side converges thanks to the bound on the Jacobian and to the bounded compressibility of $X$:
$$\Big| \int_{\R^d}(\phi_n(X)-\phi(X)) JX \, dx \Big|
\leq e^L \int_{\R^d}|\phi_n(X)-\phi(X)| \, dx
\leq C e^L \int_{\R^d}|\phi_n-\phi| \, dx,
$$
hence we obtain \eqref{eqn:c-d-var} with $\phi$.

Finally, \eqref{eqn:JX} and \eqref{eqn:JX-1} are easily checked by direct computation and using the fact that $JX$ is absolutely continuos in the time variable.
\end{proof}

\begin{proof}[Proof of Theorem~\ref{thm:main}, Existence]
Let $\beta: \R \to \R$ be a function satisfying \eqref{defn:beta}.
From the expression \eqref{eqn:cont-damp-sol} we compute an equation involving $\beta(u(t,x))$. 
Let $\varphi\in C^\infty_c$ be a test function. By the change of variable formula \eqref{eqn:c-d-var} applied with $\phi (x)= \varphi(x) \beta(u(t,x)) $  we have that
\begin{equation*}
\int_{\R^d} \varphi(x) \beta(u(t,x)) \, dx 
=
 \int_{\R^d} \varphi(X(t,x)) \beta(u(t,X(t,x))) JX(t,x) \, dx .
\end{equation*}

Thanks to the absolute continuity of $X(\cdot, x)$, of  $JX (\cdot, x) $, and of $1/JX(\cdot,x )$ and since the set of bounded, absolutely continuous functions is an algebra, for every $x\in \R^d$ the function $t \to  \varphi(X(t,x)) \beta(u(t,X(t,x))) JX(t,x) $ is absolutely continuous. Its derivative can be computed by the explicit formula for $u$ given in \eqref{eqn:cont-damp-sol-thm} thanks to \eqref{eqn:JX} and \eqref{eqn:JX-1}: for a.e. $s\in [0,T]$
\begin{equation*}
\begin{split}
 &\partial_s \big[ \varphi(X(s,x)) \beta(u(s,X(s,x))) JX(s,x) \big]=
 \nabla \varphi(X)\cdot b(s,X) \beta(u(s,X)) JX
\\
&+ \varphi(X)  \beta'(u(s,X)) \partial_s\Big[\frac{u_0}{JX}\exp\Big({\int_0^t c(\tau, X(\tau)) \, d\tau}\Big) \Big] 
JX+\varphi(X) \beta(u(s,X)) \partial_sJX
\\
&=
 \nabla \varphi(X)\cdot b(s,X) \beta(u(s,X)) JX+ \varphi(X)  \beta'(u(s,X)) \frac{u_0}{JX}\exp\Big({\int_0^t c(\tau, X(\tau)) \, d\tau}\Big)c(s,X)JX
\\
&+
\varphi(X)  \beta'(u(s,X)) \partial_s \Big[\frac{1}{JX}\Big] {u_0} \exp\Big({\int_0^t c(\tau, X(\tau) \, d\tau}\Big) JX+\varphi(X) \beta(u(s,X)) \partial_sJX
\\
&=
 \nabla \varphi(X)\cdot b(s,X) \beta(u(s,X)) JX+ \varphi(X)  \beta'(u(s,X)) u(s,X) c(s,X) JX
\\
&+
\varphi(X)  \beta'(u(s,X)) u(s,X) \partial_s \Big[\frac{1}{JX}\Big] JX^2+\varphi(X) \beta(u(s,X)) \partial_sJX
\\
& =\big[ \nabla \varphi(X)\cdot b(s,X) \beta(u(s,X))+ \varphi(X)  \beta'(u(s,X)) u(s,X) c(s,X)
\\
&-\varphi(X)  \beta'(u(s,X)) u(s,X) \nabla \cdot b (s,X)+\varphi(X) \beta(u(s,X))  \nabla \cdot b (s,X) \big]JX
\end{split}
\end{equation*}
(for the sake of brevity we write $X$ in place of $X(s,x)$ and $JX$ in place of $JX(s,x)$).
Hence, by Fubini theorem and by the change of variable \eqref{eqn:c-d-var}, we have that
\begin{equation*}
\begin{split}
&\int_{\R^d} \varphi(x) \beta(u(t,x)) \, dx-\int_{\R^d}\varphi(x)\beta(u_0(x))\,dx\\
&=
\int_{\R^d} \int_0^t \partial_s \big[ \varphi(X(s,x)) \beta(u(s,X(s,x))) JX(s,x) \big]\, ds \, dx 
\\
&=
\int_0^t
\int_{\R^d} \Big[ \nabla \varphi(X)\cdot b(s,X) \beta(u(s,X))
+ \varphi(X)  \beta'(u(s,X)) u(s,X) c(s,X)
\\
&-\varphi(X)  \beta'(u(s,X)) u(s,X) \nabla \cdot b (s,X) +\varphi(X) \beta(u(s,X))  \nabla \cdot b (s,X) \Big] JX
 \, dx \, ds
 \\
&=\int_0^t\int_{\R^d} \Big[\nabla\varphi(x)\cdot b(s,x)\beta(u(s,x)) + \varphi(x)c(s,x)u(s,x) \beta'(u(s,x))
\\
&
+ \varphi(x)\nabla \cdot b(s,x) \big(-u(s,x) \beta'(u(s,x))+\beta(u(s,x))\big) \Big] \, dx \, ds.
\end{split}
\end{equation*}
Notice that the integrand in the right-hand side is in $L^1((0,T)\times \R^d))$ thanks to the properties of $\beta$ and since $\varphi$ is compactly supported.
We have therefore verified that the function $t \to \int_{\R^d} \varphi(x) \beta(u(t,x)) \, dx $ is absolutely continuous in $[0,T]$ and that \eqref{defn:renorm-sol-tested} holds; we conclude that $u$ is a renormalized solution thanks to Remark~\ref{rmk:renorm-operativa}.
\end{proof}

\section{Uniqueness}\label{sec:un}
In this section we are going to prove the uniqueness part of Theorem~\ref{thm:main}. In Lemma~\ref{lemma:diff-renorm} we prove that under our assumptions the difference of renormalized solutions is still a renormalized solution following the lines of  \cite[Lemma II.2]{lions}. Therefore, to prove uniqueness in Theorem~\ref{thm:main} it is enough to show that every renormalized solution starting from $u_0=0$ is identically $0$. 
The following simple lemma states the property of the particular renormalization function which allows to pass to the limit in the damping term.

\begin{lemma}\label{lemma:arctan}
Let $\beta(r) = \arctan (r): \R \to (-\pi/2, \pi/2)$ and, for every $M>0$, let $\beta_M(r) = M\beta(r/M)$. Then we have that
\begin{equation}
\label{est:arctan}
|r_1 \beta'_M(r_1)- r_2 \beta'_M(r_2)| \leq | \beta_M(r_1)- \beta_M(r_2)| \qquad \forall r_1, r_2 \in \R.
\end{equation}
\end{lemma}
\begin{proof}
First we prove the inequality for $M=1$, namely 
\begin{equation}
\label{est:arctan-1}
\Big|\frac{r_1}{1+r_1^2}- \frac{r_2}{1+r_2^2} \Big| \leq | \arctan(r_1)- \arctan(r_2)| \qquad \forall r_1, r_2 \in \R.
\end{equation}
Setting $t_i = \arctan(r_i)$, $i=1,2$, the inequality is equivalent to
$$\Big|\frac{\tan(t_1)}{1+\tan^2(t_1)}- \frac{\tan(t_2)}{1+\tan^2(t_2)} \Big| \leq | t_1- t_2| \qquad \forall t_1,t_2 \in \Big( -\frac{\pi}{2}, \frac{\pi}{2}\Big).$$
Since the left-hand side can be rewritten as $|\sin(2t_1)/2- \sin(2t_2)/2|$ and the function $\sin(2t)/2$ is Lipschitz with constant $1$, the previous inequality is satisfied.
To prove \eqref{est:arctan} with $M>0$, we apply \eqref{est:arctan-1} at $r_1/M$ and $r_2/M$ to obtain
$$
\Big|\frac{Mr_1}{M^2+r_1^2}- \frac{Mr_2}{M^2+r_2^2}\Big| \leq \Big| \arctan\Big(\frac{r_1}{M}\Big)- \arctan\Big(\frac{r_2}{M}\Big)\Big| \qquad \forall r_1, r_2 \in \R.
$$
Multiplying both sides by $M$ we obtain \eqref{est:arctan}.
\end{proof}

\begin{lemma}\label{lemma:diff-renorm}
Let $b\in L^1((0,T); BV_{\rm loc} (\R^d; \R^d))$ be a vector field with $\nabla \cdot b \in L^1((0,T); L^1_{\rm loc}(\R^d))$. Let $c\in L^1_{\rm loc}( (0,T) \times \R^d)$
and let $u_0:\R^d \to \R$ be a measurable function.
Let $u_1$ and $u_2$ be renormalized solutions of \eqref{eqn:cont-damp} with initial datum $u_0$.

Then $u:= u_1-u_2$ is a renormalized solution with initial datum $0$.
\end{lemma}
\begin{proof}
Let $M>0$ and $\beta_M(r) = M\arctan(r/M)$ for every $r\in \R$. Notice that $\beta_M$ satisfies \eqref{defn:beta}, so that in the sense of distributions
$$ \partial_t \beta_M (u_i) + \nabla \cdot (b\beta_M(u_i)) + \nabla \cdot b \big(u_i \beta_M'(u_i)-\beta_M(u_i)\big) = cu_i \beta_M'(u_i) \qquad i=1,2.
$$
Taking the difference between these equations and setting $v_M = \beta_M(u_1)-\beta_M(u_2)$ we obtain that $v_M$ solves in the sense of distributions
$$\partial_t v_M + \nabla \cdot (bv_M) = (c-\nabla \cdot b) [ u_1 \beta_M'(u_1)- u_2 \beta_M'(u_2)] +\nabla \cdot b  \, v_M.
$$
Thanks to the assumptions on $b$, since the right hand side of the previous equation is locally integrable, and since $v_M \in L^\infty((0,T)\times \R^d)$, it follows by \cite{ambrosio} (see also \cite[Theorem 35]{amcri-edi}) that $v_M$ is also a renormalized solution, namely for every $\gamma$ which satisfies \eqref{defn:beta} we have
$$\partial_t \gamma(v_M) + \nabla \cdot (b\gamma(v_M)) = (c-\nabla \cdot b)\gamma'(v_M)v_M \frac{ u_1 \beta_M'(u_1)- u_2 \beta_M'(u_2)}{\beta_M(u_1)- \beta_M(u_2)} +\nabla \cdot b  \, \gamma(v_M).
$$
This means that, since $v_M(0,\cdot) =0$, for every $\phi \in C^\infty_c([0,T) \times \R^d)$ we have 
 \begin{equation}
 \label{eqn:renorm-sol-M}
 \begin{split}
&
 -\int_0^T\int_{\R^d} [\partial_t \phi + \nabla \phi\cdot b ]\gamma(v_M) \, dx \, dt =
  \\&\int_0^T \int_{\R^d} \phi \Big[ (c-\nabla \cdot b)\gamma'(v_M)v_M \frac{ u_1 \beta_M'(u_1)- u_2 \beta_M'(u_2)}{\beta_M(u_1)- \beta_M(u_2)} +\nabla \cdot b  \, \gamma(v_M) \Big] \, dx \, dt.
  \end{split}
 \end{equation}
Then, we let $M$ go to $\infty$ in the previous equation. First, we note that since $\beta_M(r) \to r$ as $M\to \infty$ it follows that  $v_M$ converges to $u_1-u_2$ pointwise as $M\to \infty$. 
As regards the left-hand side of \eqref{eqn:renorm-sol-M}, $\gamma(v_M)$ converges pointwise to $\gamma (u_1-u_2)$ and these functions are bounded by $\|\gamma \|_{\infty}$. 
The right-hand side of \eqref{eqn:renorm-sol-M} converges pointwise to the right-hand side of \eqref{eqn:renorm-sol-nM} below and by Lemma~\ref{lemma:arctan} it is bounded by the $L^1_{\rm loc}$ function $(|c|+2 |\nabla \cdot b|) \|z\gamma'(z)\|_{L^\infty(\R^d)}$. Hence by dominated convergence we get
\begin{equation}\label{eqn:renorm-sol-nM} 
-\int_0^T\int_{\R^d} [\partial_t \phi + \nabla \phi\cdot b ]\gamma(u) \, dx \, dt=
\int_0^T \int_{\R^d} \phi(t,x) \Big[ (c-  \nabla \cdot b )u \gamma'(u) +  \nabla \cdot b \gamma(u)
 \Big] \, dx \, dt
 \end{equation}
for every $\gamma$ which satisfies \eqref{defn:beta}.
\end{proof}

In the following lemma we enlarge the class of admissible test functions in \eqref{defn:renorm-sol-tested-st}. 
As it will be clear from the proof of Theorem~\ref{thm:main}, a particular Lipschitz, decaying test function will play an important role. In particular in the proof of the uniqueness in Theorem~\ref{thm:main} the estimate \eqref{eqn:est-1} fails when only compactly supported smooth test functions are considered.
\begin{lemma}\label{lemma:test-fund} 
Let $C>0$ and let  $b$ and  $u_0$ be as in Theorem~\ref{thm:main}. Let $u$ be a renormalized solution of \eqref{eqn:cont-damp}
and let $\varphi \in W^{1,\infty}(\R^d)$ be a function with the following decay
 \begin{equation}
 \label{defn:decay-test}
|\varphi(x)| \leq \frac{C}{(1+|x|)^{d+1}},
\qquad |\nabla \varphi(x)| \leq \frac{C}{(1+|x|)^{d+2}} \qquad a.e.\,\,\,\, x\in \R^d.
\end{equation}
Then the function $\int_{\R^d} \varphi(x) \beta (u(t,x))\, dx
$ coincides a.e. with an absolutely continuous function $\Gamma(t)$ such that
$\Gamma(0)=\int_{\R^d} \varphi(x) \beta (u_0(x))\, dx $
   and for a.e. $t\in [0,T]$
 \begin{equation}
 \label{defn:renorm-sol-tested-many-functions}
 \begin{split}
 \frac{d}{dt} \Gamma(t) 
 = \int_{\R^d} \nabla\varphi \cdot b\beta(u) \, dx  +
\int_{\R^d} \varphi \Big[ \nabla \cdot b \big(\beta(u)-u \beta'(u)\big) +  cu \beta'(u) \Big] \, dx .
\end{split}
\end{equation}
\end{lemma}
\begin{proof}
Although the proof is a standard argument via approximation, we sketch it for the sake of completeness. We approximate the function $\varphi$ by means of smooth, compactly supported functions $\varphi_n$ satisfying the same decay \eqref{defn:decay-test} with $C$ independent on $n$. 
By Remark~\ref{rmk:renorm-operativa}, the function $t \to \int_{\R^d} \varphi_n(x) \beta (u(t,x))\, dx $ coincides for a.e. $t \in [0,T]$ with an absolutely continuous function $\Gamma_n(t)$ which satisfies \eqref{defn:renorm-sol-tested} and $\Gamma_n(0) = \int_{\R^d} \varphi_n(x) \beta (u_0(x))\, dx $.
Thanks to  \eqref{defn:decay-test}, to the growth assumptions on $b$, and to the integrability of $c$, by dominated convergence we get that 
 \begin{equation}
 \label{defn:conv-gamma-deriv}
 \begin{split}
 \lim_{n\to \infty} \frac{d}{dt} \Gamma_n(t) &= 
 \lim_{n\to \infty} \int_{\R^d} \nabla\varphi_n \cdot b\beta(u) \, dx  +
\int_{\R^d} \varphi_n \Big[ \nabla \cdot b \big(\beta(u)-u \beta'(u)\big) +  cu \beta'(u) \Big] \, dx 
\\
&
= \int_{\R^d} \nabla\varphi \cdot b\beta(u) \, dx  +
\int_{\R^d} \varphi \Big[ \nabla \cdot b \big(\beta(u)-u \beta'(u)\big) +  cu \beta'(u) \Big] \, dx 
 \end{split}
\end{equation}
in $L^1(0,T)$.
Moreover by dominated convergence we have
$$\lim_{n\to \infty}  \Gamma_n(0) = \int_{\R^d} \varphi(x) \beta (u_0(x))\, dx
$$
and for a.e. $t\in [0,T]$
 \begin{equation}
 \label{defn:conv-gamma}
\Gamma(t) = \lim_{n\to \infty} \Gamma_n(t) = \lim_{n\to \infty} \int_{\R^d} \varphi_n(x) \beta (u(t,x))\, dx = \int_{\R^d} \varphi(x) \beta (u(t,x))\, dx.
\end{equation}
Hence the functions $\Gamma_n$ pointwise converge to an absolutely continuous function $\Gamma:[0,T] \to \R$ such that  \eqref{defn:renorm-sol-tested-many-functions} holds, $\Gamma (0)= \int_{\R^d} \varphi(x) \beta (u_0(x))\, dx$, and $\Gamma(t) = \int_{\R^d} \varphi(x) \beta (u(t,x))\, dx$ for a.e. $t\in [0,T]$.
\end{proof}

\begin{proof}[Proof of Theorem~\ref{thm:main}, Uniqueness]
Up to taking the difference of two renormalized solutions, which is still a renormalized solution with initial datum $0$ by Lemma~\ref{lemma:diff-renorm}, it is enough to show that if $u$ is a renormalized solution with initial datum $0$ then $u = 0$ in $[0,T] \times \R^d$.

Let $\delta>0$. We consider the positive function
\begin{equation}
\label{defn:beta-delta}
\beta_\delta(r) = \log\Big(1+\frac{[\arctan(r)]^2}{\delta} \Big) \qquad \forall r\in \R,
\end{equation}
which satisfies \eqref{defn:beta} and in particular, thanks to \eqref{est:arctan} applied with $M=1$, $r_1=r$, $r_2=0$
\begin{equation}
\label{eqn:propr-gamma-de}
|r \beta'_\delta(r)| = \Big| \frac{\arctan(r)}{\delta+ [\arctan(r)]^2} r\arctan'(r) \Big| \leq 1 \qquad \forall r\in \R.
\end{equation}
For every $R>0$ consider
\begin{equation}
\label{defn:phiR}
\varphi_R(x) =
\begin{cases}
\displaystyle{\frac{1}{2^{d+1}} }\qquad & x\in \R^d, \, |x|<R
\\
\displaystyle{\frac{R^{d+1}}{(R + |x|)^{d+1}} } \qquad & x\in \R^d, \, |x|>R
.
\end{cases}
\end{equation}
We use $\beta_\delta$ to renormalize the solution $u$ and $\varphi_R$ as a test function.
Notice that $\varphi_R \in L^1\cap W^{1,\infty}(\R^d)$ with $0\leq \varphi_R \leq 1$ and by Lemma~\ref{lemma:test-fund} the function $\varphi_R$ is an admissible test function in \eqref{defn:renorm-sol-tested-many-functions}.
Hence there exists an absolutely continuous function $\Gamma_{\delta,R} : [0,T] \to \R$ such that $\Gamma_{\delta,R}(0) =0$ and for a.e. $t\in [0,T]$
$$\Gamma_{\delta,R}(t) = \int_{\R^d} \varphi_R(x) \beta_\delta (u(t,x))\, dx,$$
 \begin{equation}
 \label{defn:renorm-sol-tested-2}
 \begin{split}
 \frac{d}{dt} \Gamma_{\delta,R}(t)
  = \int_{\R^d} \nabla\varphi_R \cdot b\beta_\delta(u) \, dx  +
\int_{\R^d} \varphi_R (c- \nabla \cdot b) u \beta_\delta'(u)  \, dx 
+\int_{\R^d} \varphi_R \nabla \cdot b \beta_\delta(u) \, dx 
\end{split}
\end{equation}
(here and in the following we omit the dependence of $b$, $c$, $u$ on $(t,x)$ and of $\varphi_R$ on $x$).
We estimate each term in the right-hand side of \eqref{defn:renorm-sol-tested-2}.
The third term can be estimated thanks to the condition on the divergence of $b$
\begin{equation}
\label{eqn:est-3}
\int_{\R^d} \varphi_R \nabla \cdot b \beta_\delta(u) \, dx
\leq \|\nabla \cdot b(t,\cdot)\|_{L^\infty(\R^d)} \int_{\R^d} \varphi_R \beta_\delta(u) \, dx.
\end{equation}
As regards the second term, we use \eqref{eqn:propr-gamma-de} to deduce
\begin{equation}
\label{eqn:est-2}
\int_{\R^d} \varphi_R (c- \nabla \cdot b) u \beta_\delta'(u)  \, dx 
\leq \int_{\R^d} \varphi_R (|c|+ |\nabla \cdot b|)  \, dx 
\leq \int_{\R^d} |c|\, dx +
  \|\nabla \cdot b(t,\cdot)\|_{L^\infty(\R^d)} \int_{\R^d} \varphi_R  \, dx .
\end{equation}
To estimate the first term, we take into account the growth condition \eqref{eqn:growth-b} on $b$. Let $b_1$ and $b_2$ two nonnegative functions such that
$$
\frac{|b(t,x)|}{1+|x|}\leq b_1(t,x)+b_2(t), \qquad b_1 \in L^1((0,T)\times\R^d), \qquad 
b_2 \in L^1((0,T)).
$$
Notice that $\nabla \varphi_R(x)$ can be explicitly computed;  for every $x\in \R^d$ with $|x|<R$ it is $0$ and if $|x|>R$ we have that $|\nabla \varphi_R(x)|\leq (d+1) \varphi_R(x) (R+|x|)^{-1}$. If $R>1$, we have 
\begin{equation}
\label{eqn:est-1}
\begin{split}
\int_{\R^d} \nabla\varphi_R \cdot b\beta_\delta(u) \, dx 
&\leq
(d+1) \int_{\R^d \setminus B_R}\frac{\varphi_R}{R+|x|} (1+|x|) (b_1+b_2) \beta_\delta(u) \, dx 
\\
&\leq
(d+1) \int_{\R^d \setminus B_R}{\varphi_R} (b_1+b_2) \beta_\delta(u) \, dx 
\\
&\leq
(d+1) \log\Big(1+\frac{\pi^2}{4\delta}\Big) \int_{\R^d \setminus B_R} b_1\, dx 
+
(d+1) b_2 \int_{\R^d}{\varphi_R} \beta_\delta(u) \, dx .
\end{split}
\end{equation}

Setting for every $t\in [0,T]$ the $L^1$ functions:
$$a(t) =  \|\nabla \cdot b(t,\cdot)\|_{L^\infty(\R^d)} +(d+1) b_2(t),$$
$$b_R(t) =\|c(t, \cdot) \|_{L^1(\R^d)} + \|\nabla \cdot b(t,\cdot)\|_{L^\infty(\R^d)} \| \varphi_R  \|_{L^1(\R^d)},$$
 $$c_R(t) = (d+1) \|b_1(t,\cdot) \|_{L^1(\R^d \setminus B_R)},
 $$
 from \eqref{defn:renorm-sol-tested-2}, \eqref{eqn:est-3}, \eqref{eqn:est-2}, and \eqref{eqn:est-1} we deduce that for a.e. $t\in [0,T]$
 $$\frac{d}{dt}\Gamma_{\delta,R}(t)  \leq a(t) \Gamma_{\delta,R}(t) + b_R(t) + c_R(t)\log\Big(1+\frac{\pi^2}{4\delta}\Big).$$
Since $\Gamma_{\delta,R}(0)=0$, by Gronwall lemma 
we obtain that for every $t\in [0,T]$
\begin{equation}
\label{est:gronw}
\begin{split}
\Gamma_{\delta,R}(t)  
&\leq \exp\Big(\int_0^T a(s)ds\Big)\Big(\int_0^T b_R(s)ds + \log\Big(1+\frac{\pi^2}{4\delta}\Big) \int_0^T c_R(s)ds \Big)
\\
&=\exp(A)\Big(B_R + \log\Big(1+\frac{\pi^2}{4\delta}\Big)C_R \Big).
\end{split}
\end{equation}
Notice that by definition 
\begin{equation}
\label{eqn:CR}
\lim_{R\to \infty} C_R= (d+1) \lim_{R\to \infty} \int_0^T  \int_{\R^d \setminus B_R} b_1(s,x)\, dx ds=0.
\end{equation}

We conclude finding a contradiction as in \cite{crippade1,boucrising}.
Let us assume that $u(t,\cdot)$ is not identically $0$ for some $t\in [0,T]$; then $\arctan u(t,\cdot)$ is not identically $0$ and there exists $R_0>0$ and $\gamma>0$ such that $\Leb{d}(\{ x \in B_{R_0}: [\arctan u(t,x)]^2> \gamma\})>0$.
Dividing \eqref{est:gronw} by $\log(1+\gamma/\delta^2)$ we obtain that for every $R\geq R_0$
\begin{equation*}
\begin{split}
0 &< \frac{\Leb{d}(\{ x \in B_{R_0}: [\arctan u(t,x)]^2> \gamma\}) }{2^{d+1}}
\leq \Big(\log\Big(1+\frac{\gamma}{\delta} \Big) \Big)^{-1}\Gamma_{\delta,R}(t) 
\\
&\leq \exp(A)\Big(\log\Big(1+\frac{\gamma}{\delta} \Big) \Big)^{-1}\Big(B_R + \log\Big(1+\frac{\pi^2}{4\delta}\Big)C_R \Big).
\end{split}
\end{equation*}
Letting $\delta$ go to $0$ we find
\begin{equation*}
\begin{split}
0 < \frac{\Leb{d}(\{ x \in B_{R_0}: [\arctan u(t,x)]^2> \gamma\}) }{2^{d+1}}
\leq \exp(A)C_R,
\end{split}
\end{equation*}
which is a contradiction thanks to \eqref{eqn:CR} provided that $R$ is chosen big enough.
\end{proof}
{\section{Divergence in $BMO$}\label{sec:bmo}}
In a recent paper \cite{bmo}, the author proved existence and uniqueness of solutions of the transport equation when the divergence $\nabla \cdot b$ of the vector field $b$ is the sum of a function in $L^\infty$ and a compactly supported function of bounded mean oscillation (defined in the sequel). This result extends the previous theory of \cite{lions,ambrosio} (see also \cite{amcri-edi} and the references quoted therein), where $\nabla \cdot b$ is assumed to be bounded in space.
Uniqueness is the most delicate point and its proof is based on a new inequality for $BMO$ functions, which gives the differential inequality $\frac{d}{dt} \Gamma \leq \Gamma\log(1+\Gamma)$ in $[0,T]$. Then, uniqueness follows by Gronwall Lemma.  
We give a different proof of uniqueness, under general growth conditions on the vector field, which follows the lines of Theorem~\ref{thm:main}.

Before stating the result, we recall the definition and the properties of functions of bounded mean oscillation.
\begin{definition}
Given a locally integrable function $f :\R^d \to \R$ we consider its average on $B_r(x)$
$$(f)_{B_r(x)} = \negint_{B_r(x)} f(y) \, dy$$
and its mean oscillation in $B_r(x)$
$$\negint_{B_r(x)} |f(y) - (f)_{B_r(x)} | \, dy.$$
We say that a function is of bounded mean oscillation (BMO) if
\begin{equation}
\label{eqn:bmp-norm}
\sup_{r>0, x\in \R^d} \negint_{B_r(x)} |f(x) - (f)_{B_r(x)} | \, dx <\infty.
\end{equation}
\end{definition}
A natural norm $\| \cdot \|_{*}$ on the quotient space of $BMO$ functions modulo the space of constant functions is given by the quantity in \eqref{eqn:bmp-norm}.
With a slight abuse of notation, for every $M>0$ we denote by $BMO_c(B_M)$ the space of $BMO$ functions $f:\R^d \to \R$ whose support is contained in $B_M$, endowed with the norm $\| \cdot \|_{*}$.

The John-Nirenberg inequality \cite{jn} implies that there exist two constants $C_{JN},c_{JN}>0$ depending only on the dimension such that for every function of bounded mean oscillation $f:\R^d \to \R$ and for every $M>0$
\begin{equation}
\label{eqn:bmo-jonir}
\Leb{d}\big( \{x\in B_M:  |f - ( f )_{B_{M}}| >\eta \} \big) \leq C_{JN} \Leb{d}(B_M) \exp\Big(- \frac{c_{JN} \eta}{ \|f\|_{*}} \Big) \qquad \forall \eta>0.
\end{equation}

In the following lemma we present the properties of $BMO$ functions which are used in the proof of Theorem \ref{thm:bmo}; in particular, we prove the exponential decay of the integral of $f$ on its superlevels.
\begin{lemma}\label{lemma:bmo-dec}
Let $f:\R^d \to \R$ be a nonnegative function of bounded mean oscillation supported in $B_M$. Then there exist $C, c>0$ depending only on $d$ and $M$ such that
\begin{equation}
\label{eqn:bmo-average}
( f)_{B_{M}} \leq 2^{d+1} \|f\|_{*},
\end{equation}
\begin{equation}
\label{eqn:bmo-est}
\int_{\R^d} {\big( f(x) - \lambda \| f \|_{*} \big)_+} \, dx \leq C \exp(-c\lambda)\| f \|_{*} \qquad \forall \lambda>2^{d+2}.
\end{equation}
\end{lemma}
\begin{proof}
Since $f=0$ on $B_{2M}\setminus B_M$ we have that
$(f)_{B_{2M}} = {2^{-d}}{(f)_{B_M}}$.
Hence 
\begin{equation*}
\begin{split}
(f)_{B_M} &= 
\frac{ ( f)_{B_{M}} - ( f)_{B_{2M}}}{1-2^{-d}}
\leq 
2 \negint_{B_{M}} | f(x) - ( f)_{B_{2M}}| \, dx 
\\
&\leq 
2^{d+1} \negint_{B_{2M}} | f(x) - ( f)_{B_{2M}} | \, dx 
\leq 
2^{d+1} \|f\|_{*},
\end{split}
\end{equation*}
which proves \eqref{eqn:bmo-average}.
Thanks to \eqref{eqn:bmo-average}, for every $\lambda >2^{d+2}$ we have that 
$$\big( f(x) - \lambda \| f \|_{*} \big)_+ \leq \big( f(x) - (f)_{B_M} \big)_+ \qquad \forall x\in \R^d$$
 and similarly
$$\{ x\in B_M : f(x)> \lambda \| f \|_{*} \} \subseteq \Big\{ x\in B_M : f(x)- (f)_{B_M}> \frac{\lambda \| f \|_{*}}{2} \Big\}.$$
 Using also \eqref{eqn:bmo-jonir}, we deduce that 
\begin{equation}
\label{eqn:bmo-sotto}
\begin{split}
 \int_{B_{M}}{\big( f(x) - \lambda \| f \|_{*} \big)_+} \, dx 
 &\leq 
 \int_{  \{x\in B_M : f- ( f)_{B_{M}}>  \lambda \| f \|_{*}/2\}
 }{\big( f(x) - ( f)_{B_{M}} \big)_+} \, dx 
 \\
 &=
 \int_{\lambda\| f \|_{*}/2}^\infty { \Leb{d}\big( \{ x\in B_{M}: f(x) - ( f)_{B_M}> r \} \big)}\, dr
 \\
 &\leq
C_{JN}  \Leb{d}(B_M) \int_{\lambda\| f \|_{*}/2}^\infty \exp\Big(- \frac{c_{JN} r}{ \|f\|_{*}} \Big)  \, dr
 \\
 &\leq
  \frac{C_{JN}\Leb{d}(B_M) }{c_{JN}} \|f\|_{*} \exp\Big(-\frac{c_{JN} \lambda}{2}\Big),
 \end{split}
\end{equation}
which proves \eqref{eqn:bmo-est}.
\end{proof}

In the following, we prove that the continuity equation
\begin{equation}
\label{eqn:transp}
\partial_t u+ \nabla \cdot( bu) = 0
\end{equation}
 with a $BV$ vector field with divergence in $BMO_c(\R^d)+L^\infty(\R^d)$ is well posed in the class of bounded distributional solutions. We recall that a  function $u \in L^\infty( (0,T) \times \R^d)$ is a distributional solution of~\eqref{eqn:transp} with initial datum $u_0 \in L^\infty(\R^d)$ if  for every $\phi \in C^{\infty}_c ( [0,T) \times \R^d)$
 \begin{equation*}
\int_{\R^d} \phi(0,x) u_0(x)t\, dx +
 \int_0^T \! \int_{\R^d} [\partial_t \phi(t,x) + \nabla\phi (t,x) b(t,x) ]u(t,x) \, dx \, dt  \, =0.
\end{equation*}
 
\begin{theorem}\label{thm:bmo}
Let $u_0 \in L^\infty(\R^d)$, $M>0$, and $b\in L^1((0,T); BV_{\rm loc} (\R^d;\R^d))$ a vector field such that
\begin{equation}
\label{eqn:bmo-div-b}
|\nabla \cdot b| \in L^1((0,T); L^\infty(\R^d)) + L^1((0,T); BMO_c(B_M)),
\end{equation}
$$\frac{|b(t,x)|}{1+|x|} \in L^1((0,T); L^1(\R^d))+ L^1((0,T); L^\infty(\R^d)).
$$
Then there exists a unique distributional solution $u\in L^\infty( (0,T) \times \R^d)$ of \eqref{eqn:transp}
with initial datum $u_0$.
\end{theorem}
Existence is obtained through a standard regularization argument, see \cite{ambrosio} or \cite[Appendix A2]{bmo}, and we omit the proof.
The result can be also generalized adding a right-and side of the form $cu$, for some $c \in L^1((0,T)\times \R^d)$, with the same ideas as in Section~\ref{sec:un}, and we would have to consider renormalized solutions in place of distributional solutions.

\begin{proof}[Proof of uniqueness]
Given $R,\delta>0$, we consider the functions $\varphi_R$ and $ \beta_\delta$
defined as in~\eqref{defn:beta-delta} and~\eqref{defn:phiR}.
By the linearity of the continuity equation \eqref{eqn:transp}, up to taking the difference of two distributional solutions with the same initial datum, it is enough to show that any distributional solution $u$ with initial datum $0$ is constantly $0$. Thanks to the assumptions on $b$ and to \cite{ambrosio}, $u$ is also a renormalized solution with initial datum $0$ (in the sense of Definition~\ref{defn:renorm-sol} with $c=0$).
By an easy adaptation of Lemma~\ref{lemma:test-fund} (with $c=0$ and the assumption \eqref{eqn:bmo-div-b} on $\nabla \cdot b$) we can use $\varphi_R$ as a test function in \eqref{defn:renorm-sol-tested-many-functions}; in other words, for every $R,\delta>0$ there exists an absolutely continuous function $\Gamma_{\delta,R}:[0,T] \to \R$ such that $\Gamma_{\delta,R}(0)=0$,
$$\Gamma_{\delta,R}(t) = \int_{\R^d} \varphi_R(x) \beta (u(t,x))\, dx \qquad \mbox{for a.e. }t \in [0,T],$$
 \begin{equation}
 \label{defn:bmo-renorm-sol}
 \begin{split}
 \frac{d}{dt} \Gamma_{\delta,R}
 = \int_{\R^d} \nabla\varphi_R b\beta_\delta(u) \, dx  +
\int_{\R^d} \varphi_R \nabla \cdot b \big(\beta_\delta(u)-u \beta_\delta'(u)\big) \, dx \qquad \mbox{a.e. in } [0,T].
\end{split}
\end{equation}

We estimate the first term in the right-hand side of \eqref{defn:bmo-renorm-sol} as in \eqref{eqn:est-1}. For every $R>1$ 
\begin{equation}
\label{eqn:bmo-est-1}
\begin{split}
\int_{\R^d} \nabla\varphi_R \cdot b\beta_\delta(u) \, dx 
&\leq
(d+1) \log\Big(1+\frac{\pi^2}{4\delta}\Big) \int_{\R^d \setminus B_R} b_1\, dx 
+
(d+1) b_2 \int_{\R^d}{\varphi_R} \beta_\delta(u) \, dx ,
\end{split}
\end{equation}
where 
$b_1$ and $b_2$ are nonnegative functions such that
$$
\frac{|b(t,x)|}{1+|x|}\leq b_1(t,x)+b_2(t) \qquad b_1 \in L^1((0,T)\times\R^d), \qquad 
b_2 \in L^1((0,T)).
$$

Let $d_1$ and $d_2$ be nonnegative functions such that
$$
|\nabla \cdot b(t,x)| \leq d_1(t,x)+d_2(t,x), \quad d_1 \in L^1((0,T); L^\infty(\R^d)), \quad 
d_2 \in L^1((0,T); BMO_c(B_M)).
$$

Let $\lambda>2^{d+2}$ to be chosen later.
Since $|u \beta_\delta'(u)| \leq 1$ for every $\delta>0$ (see \eqref{eqn:propr-gamma-de}) we estimate the second term in the right-hand side of \eqref{defn:bmo-renorm-sol}
\begin{equation}
\label{eqn:bmo-est-2}
\begin{split}
&\int_{\R^d} \varphi_R \nabla \cdot b \big(\beta_\delta(u)-u \beta_\delta'(u)\big) \, dx 
\leq \int_{\R^d} \varphi_R (d_1+d_2) \big|\beta_\delta(u)-u \beta_\delta'(u)\big| \, dx 
\\
&\leq 
(\|  d_1 \|_{\infty} + \lambda \|  d_2 \|_{*}) \int_{\R^d} \varphi_R \big(\beta_\delta(u)+1\big) \, dx 
+
\Big( \log\Big(1+\frac{\pi^2}{4\delta}\Big)+1\Big)
\int_{\R^d}  \big(|d_2 | - \lambda  \|  d_2 \|_{*}\big)_+ \, dx 
\\
&\leq 
(\|  d_1 \|_{\infty} + \lambda \|  d_2 \|_{*}) \int_{\R^d} \varphi_R \big(\beta_\delta(u)+1\big) \, dx 
+
C \Big( \log\Big(1+\frac{\pi^2}{4\delta}\Big)+1\Big) \exp(-c\lambda) \|d_2 \|_{*},
\end{split}
\end{equation}
where in the last inequality we applied Lemma~\ref{lemma:bmo-dec} to the function $d_2(t, \cdot)$.

For every $t\in [0,T]$ we define the functions
$$a_\lambda(t) =  \|  d_1(t, \cdot) \|_{L^\infty(\R^d)} + \lambda \|  d_2(t,\cdot) \|_{*}  + (d+1) b_2(t),$$
$$b_{\lambda, R}(t) = \big(\|  d_1(t, \cdot) \|_{L^\infty(\R^d)} + \lambda \|  d_2(t,\cdot) \|_{*}\big) \| \varphi_R \|_{L^1(\R^d)} + C \exp(-c\lambda) \|d_2 (t,\cdot) \|_{*},
$$
 $$c_{R}(t) = (d+1) \|b_1(t,\cdot) \|_{L^1(\R^d \setminus B_R)},$$
 $$d_{\lambda}(t) = C\exp(-c\lambda) \|d_2 (t,\cdot)  \|_{*} .
 $$
From \eqref{defn:bmo-renorm-sol}, \eqref{eqn:bmo-est-1}, and \eqref{eqn:bmo-est-2} we deduce that
 $$\frac{d}{dt}\Gamma_{\delta,R}(t)  \leq a_\lambda(t) \Gamma_{\delta,R}(t) + b_{\lambda,R}(t) + (c_R(t)+ d_\lambda(t)) \log\Big(1+\frac{\pi^2}{4\delta}\Big).$$
 
 Let $\tau_0>0$ to be chosen later in terms of $b$ and independent on $R, \lambda$. It is enough to show that $u(t,\cdot) = 0$ for every $t\in [0, \tau_0]$, then we can repeat the argument in the following time intervals.
Since by assumption $\Gamma_{\delta,R}(0)=0$, by Gronwall lemma for every $t\in [0,\tau_0]$
\begin{equation}
\label{est:bmo-gronw}
\begin{split}
\Gamma_{\delta,R}(t)  
&\leq \exp\Big(\int_0^{\tau_0} a_\lambda (s)ds\Big)\Big(\int_0^{\tau_0} b_{\lambda, R}(s)ds + \log\Big(1+\frac{\pi^2}{4\delta}\Big) \int_0^{\tau_0} (c_{R}(s) + d_\lambda(s)) ds \Big)
\\
&=\exp(A_\lambda )\Big(B_{\lambda, R} + \log\Big(1+\frac{\pi^2}{4\delta}\Big)(C_{R} + D_\lambda) \Big).
\end{split}
\end{equation}

If, by contradiction, $u(t,\cdot)$ is not identically $0$ for some $t\in [0,{\tau_0}]$, then there exist $R_0>0$ and $\gamma>0$ such that $m=\Leb{d}(\{ x \in B_{R_0}: [\arctan u(t,x)]^2> \gamma\})>0$.
We obtain that for every $R\geq R_0$
\begin{equation}
\label{eqn:bmo-contr}
\begin{split}
\frac{m }{2^{d+1}}\log\Big(1+\frac{\gamma}{\delta} \Big) 
&\leq \Gamma_{\delta,R}(t) \leq \exp(A_\lambda)\Big(B_{\lambda,R} + \log\Big(1+\frac{\pi^2}{4\delta}\Big)(C_{R} + D_\lambda) \Big).
\end{split}
\end{equation}
First we fix $\tau_0>0$ so that 
$$\int_0^{\tau_0}  \|  d_2(t,\cdot) \|_{*} \, ds \leq \frac{c}{2}.$$
Thanks to this choice, 
$\exp(A_\lambda ) D_{\lambda}$ decays exponentially as $\lambda$ goes to $\infty$; we choose $\lambda$ sufficiently large so that $\exp(A_\lambda ) D_{\lambda} < m/2^{d+3}$. Finally, since $C_R$ goes to $0$ as $R$ goes to $\infty$, we choose $R$ sufficiently large that
$\exp(A_\lambda ) C_{R} < m/2^{d+3}$.
Dividing  \eqref{eqn:bmo-contr}  by $\log(\delta^{-1})$ and letting $\delta$ go to $0$ we find a contradiction.
\end{proof}

\end{document}